\documentclass{article}

\usepackage[colorlinks = true,linkcolor = blue,urlcolor  = blue,citecolor = blue,anchorcolor = blue]{hyperref}
\author{%
  Julien Weibel\footnote{Institut Denis Poisson,
Universit\'{e} d'Orl\'{e}ans,
Universit\'{e} de Tours,
CNRS,
France
and
CERMICS, Ecole des Ponts, France.
    \textrm{\textbf{julien.weibel@normalesup.org}}}
}
\date{March 4, 2026}
\title{Ergodic theorem for branching Markov chains indexed by trees with arbitrary shape\footnote{This is the author accepted manuscript version. The final published version is available in the 
Journal of Applied Probability (2025), 62(3):1089--1104, \href{https://doi.org/10.1017/jpr.2025.14}{doi:10.1017/jpr.2025.14}. Copyright Applied Probability Trust for the final published version.}}

\usepackage{mathtools,amssymb,amsmath,amsthm}

\usepackage[english]{babel}
\usepackage{hyphenat}
\allowhyphens

\usepackage[utf8]{inputenc}
\usepackage[T1]{fontenc}
\usepackage[a4paper]{geometry}
\usepackage{color}

\usepackage{dsfont} 
\usepackage{enumitem} 
\usepackage{array} 
\usepackage{stmaryrd} 
\usepackage{tikz}	
\usetikzlibrary{positioning, shapes}
\usepackage{subcaption} 

\newlength{\myFigWidth}	
\newlength{\myFigSize}		
\theoremstyle{plain}
\newtheorem{Th}{Theorem}[section]

\newtheorem{assump}{Assumption}

\newtheorem{theorem}[Th]{Theorem}
\newtheorem{lemma}[Th]{Lemma}

\newtheorem{proposition}[Th]{Proposition}
\newtheorem{definition}[Th]{Definition}

\newtheorem{assumption}[assump]{Assumption}

\theoremstyle{definition}
\newtheorem{remark}[Th]{Remark}
\newtheorem{example}[Th]{Example}

\newtheorem*{remark*}{Remark}

\newcommand{\as}{a.s.\ }
\newcommand{\ie}{i.e.\ }
\newcommand{\eg}{e.g.\ }

\newcommand{\scalProd}[2]{\langle #1, #2 \rangle}

\newcommand{\inv}[1]{\frac{1}{#1}}

\newcommand{\cP}{\mathcal{P}}

\newcommand{\dgr}{d}

\newcommand{\T}{T}
\newcommand{\Tinf}{\mathbb{T}^\infty}
\newcommand{\Tthm}{\Tinf}
\newcommand{\Tintro}{\Tinf}

\newcommand{\G}{G}

\newcommand{\Prb}{\mathbb{P}}

\newcommand{\Esp}{\mathbb{E}}
\newcommand{\ind}{\mathds{1}}
\newcommand{\Var}{\text{Var}}

\newcommand{\N}{\mathbb{N}}

\newcommand{\R}{\mathbb{R}}

\newcommand{\drv}{\mathrm{d}}

\newcommand{\norm}[1]{\Vert#1\Vert}

\newcommand{\Ker}{\text{Ker}}

\newcommand{\SpaceX}{\mathcal{X}}

\newcommand{\eps}{\varepsilon}

\newcommand{\rooot}{\partial}
\newcommand{\parent}{\mathrm{p}}

\newcommand{\keywords}[1]{\textbf{Keywords:} #1}
\newcommand{\ams}[2]{\textbf{2020 Mathematics Subject Classification:} Primary #1, Secondary #2}

\begin{document}

\maketitle




\begin{abstract}
We prove an ergodic theorem for Markov chains indexed by the Ulam-Harris-Neveu tree 
	over large subsets with arbitrary shape under two assumptions:
(i) with high probability, two vertices in the large subset are far from each other
	and (ii) with high probability, those two vertices 
	have their common ancestor close to the root.
The assumption on the common ancestor can be replaced by
	some regularity assumption on the Markov transition kernel.
We verify that those assumptions are satisfied for some usual trees.
Finally, with Markov Chain Monte Carlo considerations in mind,
we prove when the underlying Markov chain is stationary and reversible
that the Markov chain, that is the line graph, 
yields minimal variance for the empirical average estimator
among trees with a given number of nodes.
In doing so, we prove that the Hosoya-Wiener polynomial 
is minimized over $[-1,1]$ by the line graph among trees of a given size.
\end{abstract}


\ams{60J05; 60J80}{60F25}    

\section{Introduction}

Branching Markov processes, 
which are a generalization of Markov chains where processes are indexed by trees,
are useful to describe the evolution and growth of a population.
Limit theorems, such as the law of large numbers (sometimes also called ergodic theorem in markovian contexts),
are important tools to study properties of a population such as the distribution of traits.
The law of large numbers for branching Markov processes has been studied
with both a discrete or continuous state space
\cite{athreyaLimitTheoremsPositive1998,athreyaLimitTheoremsPositive1998a}
for set indicator functions, and thus for continuous bounded functions
by Portmanteau theorem.
To study cellular aging, a more general version of the strong law of large numbers
for a wider class of test functions and for non-independent daughter cells was proved in \cite{GuyonLimitTheorem}.
This wider class of admissible test functions depends on 
the transition kernel of the branching Markov process considered,
\eg \cite{GuyonLimitTheorem} focuses on one example 
with an autoregressive Gaussian transition kernel
and with continuous and polynomially growing functions.
See also \cite{delmasDetectionCellularAging2010} for an extension 
of  \cite{GuyonLimitTheorem}
to bounded degree Bienaymé-Galton-Watson trees
and \cite{bansayeAncestralLineagesLimit2019} for an extension to time-varying environnement and trait-dependent offspring distribution.
In this article, we present an ergodic theorem for a wide class of test functions as in \cite{GuyonLimitTheorem},
and for branching Markov processes where reproduction is independent from individual traits
but where the genealogical tree of the population can have an arbitrary shape.
\medskip

A branching Markov process $X = (X_u, u\in\T)$ with values in a metric space $\SpaceX$
is a random process indexed by a rooted tree $\T$ with the Markov property:
sibling nodes take independent and identically distributed values 
that depend only on the value of their parent node.
Without loss of generality,
we may choose $\T$ to be the rooted Ulam-Harris-Neveu tree $\Tinf = \cup_{n\in\N} (\N^*)^n$.
See Definition~\ref{def:Markov_proc} below for a complete formal definition.

For simplicity, in this introduction we restrict ourselves to the case where the transition kernel $Q$
of the branching Markov process $X$ is 
\emph{ergodic} (resp. \emph{uniformly ergodic}),
that is 
$Q$ has a unique invariant measure $\mu$ and
for any continuous bounded function $f$ on $\SpaceX$, 
we have for all $x\in\SpaceX$ that $\lim_{n\to\infty} \vert Q^n f(x) - \scalProd{\mu}{f} \vert = 0$ 
(resp. $\lim_{n\to\infty} \sup_{x\in\SpaceX} \vert Q^n f(x) - \scalProd{\mu}{f} \vert = 0$).

For a finite (non-empty) subset $A\subset\Tintro$ and some function $f$ on $\SpaceX$,
we define the normalized empirical average:
\[ \bar M_{A}(f) = \vert {A} \vert^{-1} \sum_{u\in {A}} f(X_u) . \]
Our goal is to study the asymptotic behavior of the normalized empirical average
when the averages are performed on a sequence $(A_n)_{n\in\N}$ of finite subsets of $\Tintro$
whose size goes to infinity.
For instance, the averaging set $A_n$ can be 
the $n$-th generation $\G_n$ of a tree $\T$,
	or $\T_n$ the tree $\T$ up to generation $n$.
In the following of the article, 
we will always use this definition for $G_n$ and $T_n$.

To this end, we need a geometrical assumption on the sequence of finite subsets $(A_n)_{n\in\N}$,
which states that vertices are far away from each other with high probability.

\begin{assumption}[Geometrical]
	\label{assump:An_ergodic_theorem}
Let $A_n\subset \Tintro$ for $n\in\N$ be finite (non-empty) subsets,
and let $U_n$ and $V_n$ 
be independent and uniformly sampled elements of $A_n$.
Denoting by $\dgr$ the graph distance on $\Tinf$,
for all $k\in\N$, we have:
\begin{equation*}
\Prb( \dgr(U_n, V_n) \leq k)
= \vert A_n\vert^{-2} \sum_{u,v\in A_n} \ind_{\{ \dgr(u,v)\leq k \}}
\underset{n\to\infty}{\longrightarrow} 0.
\end{equation*}
\end{assumption}
Let us stress that Assumption~\ref{assump:An_ergodic_theorem} 
implies that $\lim_{n\to\infty} \vert A_n \vert = \infty$.

We either need to assume that $Q$ is uniformly ergodic,
or that $Q$ is ergodic and the sequence $(A_n)_{n\in\N}$ satisfies the following
condition stating that the last common ancestor of two vertices is near the root with high probability.
Denote by $h(u)$ the height of a vertex $u$ 
and by $u\land v$ the common ancestor of two vertices $u$ and $v$
(see Section~\ref{subsection:notation}).

\begin{assumption}[Ancestral]
	\label{assump:ancestor_tight}
For all $n\in\N$, let $U_n$ and $V_n$ 
be independent and uniformly sampled elements of $A_n$.

The sequence of random variables $(h(U_n\land V_n))_{n\in\N}$ is tight,
that is, for every $\eps>0$, there exists $k\in\N$ such that
$\Prb(h(U_n\land V_n) > k) < \eps$ for $n$ large enough.
\end{assumption}

Note that Assumptions~\ref{assump:An_ergodic_theorem} and~\ref{assump:ancestor_tight}
are similar to Assumptions 2.(b) and 2.(a), respectively,
considered in \cite{bansayeAncestralLineagesLimit2019}
in the case where $A_n$ is the $n$-th generation of the tree.

\begin{remark}[Some sufficient conditions for Assumptions~\ref{assump:An_ergodic_theorem} and \ref{assump:ancestor_tight}, see Section~\ref{section:sufficient_conditions}]
Assumption~\ref{assump:An_ergodic_theorem} is always satisfied for Cayley and Bethe trees and for bounded degree trees (see Lemma~\ref{lemma:assumpAr2_on_bounded_degree_trees})
as long as $\lim_{n\to\infty} \vert A_n \vert = \infty$.
Assumptions~\ref{assump:An_ergodic_theorem} and \ref{assump:ancestor_tight} are satisfied 
for spherically symmetric trees when $A_n = \G_n$
	(see Lemma~\ref{lemma:sperically_symmetric_trees}).
In Lemma~\ref{lemma:GW_Gn_dist_distrib}, 
	we prove that Assumptions~\ref{assump:An_ergodic_theorem} and \ref{assump:ancestor_tight}
	are satisfied for super-critical Bienaymé-Galton-Watson trees conditioned on non-extinction 
	when $A_n = \G_n$ or $\T_n$. 
\end{remark}

We can now formulate the ergodic theorem for branching Markov processes on trees with arbitrary shape.
In Section~\ref{section:main_theorem}, we prove Theorem~\ref{thm:Ergodic_theorem_on_arbitrary_trees}, a more general version of this theorem.

\begin{theorem}[Ergodic theorem for Markov processes on trees with arbitrary shape]
	\label{thm:intro:Ergodic_theorem_on_arbitrary_trees}
Let $(A_n)_{n\in\N}$ be a sequence of finite subsets of $\Tintro$
	that satisfies Assumption~\ref{assump:An_ergodic_theorem}.
Let $X$ be a branching Markov process indexed by $\Tintro$ with values in $\SpaceX$
	whose transition kernel $Q$ is ergodic.
Assume that either $Q$ is uniformly ergodic or that $(A_n)_{n\in\N}$ satisfies Assumption~\ref{assump:ancestor_tight}.
Then, for every continuous bounded function $f$ on $\SpaceX$, we have: 
\[ \bar M_{A_n}(f) = \vert {A_n} \vert^{-1} \sum_{u\in {A_n}} f(X_u) \overset{L^2}{\underset{n\to\infty}\longrightarrow}  \scalProd{\mu}{f}
. \]
\end{theorem}

\begin{remark}
We discuss 
the main difference between Theorem~\ref{thm:intro:Ergodic_theorem_on_arbitrary_trees} 
and the law of large numbers for branching Markov process found in \cite{GuyonLimitTheorem}.
The results in \cite{GuyonLimitTheorem} apply to Markov processes
where daughter nodes can have non-independent distributions when conditioning on their mother,
whereas in our case they must be independent.
In exchange, our results allow for more flexibility on the shape of the population's genealogical tree:
for instance more flexibility on the number of children of each node (including Bienaymé-Galton-Watson trees with unbounded degree),
or even allowing the number of children of a node to grow over time
	(\eg the degree of the root can increase as $\log n$, \ie slow condensation).
Our results could also be applied to random trees with population-wise interactions
(\eg competition where reproduction rate decreases when population size increases).
Moreover, in our results, the empirical average can be performed on a wide variety of (possibly random) subsets of the tree,
and not only to the $n$-th generation.
As an example, our results could be applied to a random subset of size $\log n$ of the $n$-th generation $\G_n$
of a super-critical Bienaymé-Galton-Watson tree chosen uniformly at random.
Also note that in our results, assumptions can be verified separately for the transition kernel $Q$ and the population's genealogical tree.
\end{remark}

Lastly, motivated by Markov Chain Monte Carlo considerations,
we study in Section~\ref{section:variance} the variance of
the empirical average estimator $\bar M_{A}(f)$ and its dependence on the shape of $A$.
We perform exact variance computation in the case where the transition kernel $Q$
	induces a self-adjoint compact operator on $L^2(\mu)$,
	which proves the following proposition
	about non-asymptotic variance comparison.

\begin{proposition}[The line graph has minimal variance]
	\label{prop:intro:lemma_var_connected_sets}
Let $\mu$ be an invariant measure for $Q$,
and assume that the transition kernel $Q$ induces a self-adjoint compact operator on $L^2(\mu)$.
Let $X$ be a branching Markov process on $\Tinf$ with transition kernel $Q$ and initial distribution $\nu$.
Let $f$ be a non-constant function in $L^2(\mu)$.

When $\nu=\mu$, we have that $\Esp\left[ \bar M_{A}(f) \right] = \scalProd{\mu}{f}$
for any finite subset $A\subset \Tinf$,
and thus the empirical average estimator has no bias.
Moreover, the minimum of the map $A \mapsto \Var ( \bar M_{A}(f) )$
among subtrees of $\Tinf$ with a given cardinal $n$
is achieved by the line graph tree (\ie the Markov chain).

Furthermore, when $n\geq 5$, the line graph is the only subtree of size $n$ achieving this minimum
if and only if $f \not\in \Ker(Q) \oplus \Ker(Q - I) \oplus \Ker (Q + I)$.
\end{proposition}

Remark that as $Q$ is a Markov kernel, its spectrum as an $L^2(\mu)$-operator is a subset of $[-1,1]$.
Note that when $f \in \Ker(Q) \oplus \Ker(Q - I)$, then the value of $\Var ( \bar M_{A}(f) )$ does not depend
on the shape of the tree $A$.
Also note that when $f\in \Ker (Q + I)$, then the value of $\Var ( \bar M_{A}(f) )$ is minimal among subtrees of size $n$
when $A$ has a balanced bipartite $2$-coloring, and for $n\geq 5$, the line graph is not the only tree with a balance bipartite $2$-coloring.
\medskip

Hence, if we want to approximate $\scalProd{\mu}{f}$,
using a branching Markov chain does not improve the rate of convergence
compared to a standard Markov chain.
\medskip

The proof of Proposition~\ref{prop:intro:lemma_var_connected_sets}
relies on decomposing the function $f$ on a basis of eigenvectors of $Q$,
the problem then reduces to minimization among trees of a given size
of the Hosoya-Wiener polynomial
$H_A(\alpha) = \sum_{u,v\in A} \alpha^{d(u,v)}$ for some $\alpha\in [-1,1]$.
This leads us to prove the following lemma
stating that the line graph tree achieves this minimum.

\begin{lemma}[The line graph minimizes the Hosoya-Wiener polynomial]
\label{lemma_minimization_Hosoya_polynomial}
Let $\alpha\in [-1,1]$ and $n\in\N^*$.
Then, the minimum of the map $A \mapsto H_A(\alpha)$
among trees with given cardinal $n$ is achieved by the line graph tree.

Furthermore, when $\alpha\in (-1,0) \cup (0,1)$,
the line graph tree of size $n$ is the only tree achieving this minimum.
\end{lemma}

Note that this result was already proved for $\alpha\in[0,1]$
in \cite[Theorem~9]{casablancaDistanceEccentricSequences2019}
where the proof relies heavily on the monotonicity of the function $d \mapsto \alpha^d$
(and more general results were proved in
\cite[Theorem~2]{tianSharpBoundsNormalization2013} and
\cite[Corollary~2.8]{wagnerDistancebasedGraphInvariants2013}
for Wiener-type indices of the form $W_f(A) = \sum_{u,v\in A} f(d(u,v))$
where $f$ is a monotonic function).
Hence, the novelty of Lemma~\ref{lemma_minimization_Hosoya_polynomial}
is for the case of $\alpha\in[-1,0)$
where the function $d\mapsto \alpha^d$ is non-monotonic,
and our proof relies on considering several cases depending on the tree structure.

\section{Main theorem}\label{section:main_theorem}

\subsection{Notations}\label{subsection:notation}

Let $\Tinf = \cup_{n\in\N} (\N^*)^n$ denote the Ulam-Harris-Neveu tree,
and denote by $\rooot$ its root, that is the empty word.

Let $u\in\Tinf$ be a vertex.
If $u$ is distinct from the root, we denote by $\parent(u)$ its parent vertex.
We denote by $h(u)$ its height, \ie the number of edges separating $u$ from the root $\rooot$.
(The height of the root $\rooot$ is zero.)
For two vertices $u,v\in\Tthm$, we denote by $u\land v$ the latest common ancestor
	of $u$ and $v$,
	and by $\dgr(u,v)$ the graph-distance between $u$ and $v$ in $\Tthm$,
	that is $\dgr(u,v) = h(u) + h(v) - 2 h(u \land v)$.

Let $X = (X_u, u\in\Tthm)$ be a stochastic process with values in a metric space $\SpaceX$.
\begin{definition}[Markov process]
	\label{def:Markov_proc}
The stochastic process $X$
	is called a \emph{(branching) Markov process}
	with transition kernel $Q$ and initial distribution $\nu$
	if for any finite subtree $\T\subset \Tthm$ with $\rooot\in T$, we have:
\begin{equation*}
	\Prb \left( \bigcap_{u\in T} \{ X_u \in \drv x_u \} \right)
	 = \nu(\drv x_\rooot)\,  \prod_{u\in T\setminus\{\rooot\}} 
			Q(x_{\parent(u)}; \drv x_u)	. 
\end{equation*}
\end{definition}

We denote by $\nu Q^n$ the distribution of a vertex in the $n$-th generation.
For a (Borel) function $f$, define the function $Q f : x\in\SpaceX \mapsto \int f(y) \, Q(x;\drv y)$
when the expression makes sens.
For a measure $\mu$ and a Borel function $f$ on $\SpaceX$, we denote
	$\mu f = \scalProd{\mu}{f} = \int_{\SpaceX} f\, \drv \mu$.
Through the rest of this section, we fix $\nu$ and $Q$.

\subsection{Statement of the main result}

Firstly, we need some assumptions on the Borel function $f$ with which 
we perform the empirical averages.

\begin{assumption}[Boundedness and convergence]
	\label{assump:Guyon_simpler}
Let $f$ be a Borel function on $\SpaceX$ such that:
\begin{enumerate}[label=(\roman*)]
\item\label{assumpGuyonSimpler1}
	$\sup_{n\in\N} \nu Q^n (f^2) < \infty$,
\item\label{assumpGuyonSimpler2}
	there exists a constant $c_f\in\R$ such that
	$\lim_{n\to\infty} \nu Q^k ((Q^n f - c_f)^2) = 0$
	for all $k\in\N$. 
\end{enumerate}
\end{assumption}
Note that if $f$ satisfies Assumption~\ref{assump:Guyon_simpler}, then so does $f - c$ for any $c\in\R$,
thus we may assume that $c_f=0$ when necessary.
Also note, using Cauchy-Schwarz and Jensen's inequalities, 
that Assumption~\ref{assump:Guyon_simpler}-\ref{assumpGuyonSimpler1} implies that
$Q^n f$, $Q^n f^2$, and $Q^k (Q^n f \times Q^m f)$ (with $n,m,k\in\N$) 
are well-defined and finite $\nu$-almost everywhere and are $\nu$-integrable.

Remark that when $Q$ is ergodic, then Assumption~\ref{assump:Guyon_simpler} is satisfied
by any continuous bounded function $f$ on $\SpaceX$, 
and we get $c_f = \scalProd{\mu}{f}$, where $\mu$ is the unique invariant measure of $Q$.
Also remark that if $F$ is a subspace of Borel functions on $\SpaceX$ that satisfy
Assumptions~(i)-(vi) on pages 11-12 in \cite{GuyonLimitTheorem},
then any function $f\in F$ satisfies Assumption~\ref{assump:Guyon_simpler} with $c_f = \scalProd{\mu}{f}$.
\medskip

For a finite subset $A\subset \Tthm$ and a Borel function $f$,
we define the empirical sum $M_A(f) = \sum_{u\in A} f(X_u)$
	and the empirical average: 
\[ \bar M_A(f) = \vert A \vert^{-1} \sum_{u\in A} f(X_u) , \]
	where $\vert A\vert$ is the cardinal of the set $A$.
Let $(A_n)_{n\in\N}$ be a sequence of finite subsets of $\Tthm$
on which we perform the empirical averages in the ergodic theorem on trees with arbitrary shape.
Remind the geometrical Assumptions~\ref{assump:An_ergodic_theorem} and \ref{assump:ancestor_tight}.

Contrary to the case of the binary tree considered in \cite{GuyonLimitTheorem},
Assumption~\ref{assump:ancestor_tight} is not always satisfied for a sequence $(A_n)_{n\in\N}$
of finite subsets of $\Tthm$ with arbitrary shape
(\eg in the case of the line graph, \ie the Markov chain).
Thus, to prove the ergodic theorem for branching Markov chains,
as an alternative to the ancestral Assumption~\ref{assump:ancestor_tight},
we also consider the following conditions on the ergodicity of the transition kernel $Q$.

\begin{assumption}[Stronger ergodicity]
	\label{assump:Q_stronger_ergodic}
Assume that any of the following conditions holds:
\begin{enumerate}[label=(\roman*)]
\item\label{assump_nu_mu}
$\nu = \mu$ is an invariant measure of $Q$.
\item\label{assump_conv_TV}
There is convergence in total variation 
	$\lim_{n\to\infty} \norm{\nu Q^n - \mu}_{\text{TV}} = 0$
	to some invariant measure $\mu$ (for $Q$),
	the function $f$ is bounded,
	and we have $\lim_{n\to\infty} \mu (Q^n f -c_f)^2 = 0$,
	where $c_f$ is the same constant as in Assumption~\ref{assump:Guyon_simpler}-\ref{assumpGuyonSimpler2}.
\item\label{assump_uniformly_ergodic}
The transition kernel $Q$ satisfies a uniformly ergodic assumption (with $\mu$ as its unique invariant measure):
	there exists a non-negative Borel function $g$ on $\SpaceX$ with $\sup_{k\in\N} \nu Q^k g^2 < \infty$
	and a sequence of positive numbers $(a_n)_{n\in\N}$ that converges to zero,
	such that for all $n\in\N$, we have
		$ \vert Q^n f - \scalProd{\mu}{f} \vert \leq a_n\ g$.
\end{enumerate}
\end{assumption}
Remark (using dominated convergence with domination by $g$) that Assumption~\ref{assump:Q_stronger_ergodic}-\ref{assump_uniformly_ergodic}
implies Assumption~\ref{assump:Guyon_simpler}-\ref{assumpGuyonSimpler2} with $c_f = \scalProd{\mu}{f}$.
Also remark that when Assumption~\ref{assump:Guyon_simpler} holds and either
	Assumption~\ref{assump:Q_stronger_ergodic}-\ref{assump_nu_mu} or \ref{assump:Q_stronger_ergodic}-\ref{assump_conv_TV} holds,
	then we have $c_f = \scalProd{\mu}{f}$ in Assumption~\ref{assump:Guyon_simpler}
(indeed, 
using Jensen's inequality, we have that $(\scalProd{\mu}{f} -c_f)^2 = \limsup_{n\to\infty} (\mu Q^n f - c_f)^2
	\leq \limsup_{n\to\infty} \mu (Q^n f - c_f)^2 =0$).
\medskip

We can now formulate the ergodic theorem for branching Markov processes on trees with arbitrary shape.

\begin{theorem}[Ergodic theorem for Markov processes on trees with arbitrary shape]
	\label{thm:Ergodic_theorem_on_arbitrary_trees}
Let $(A_n)_{n\in\N}$ be a sequence of finite subsets of $\Tthm$
	that satisfies Assumption~\ref{assump:An_ergodic_theorem}.
Let $X$ be a branching Markov process on $\Tthm$
	with transition kernel $Q$ and initial distribution $\nu$.
Let $f$ be a Borel function on $\SpaceX$ that satisfies Assumption~\ref{assump:Guyon_simpler}. 
Furthermore assume that either Assumption~\ref{assump:ancestor_tight} 
	or \ref{assump:Q_stronger_ergodic} holds.
Then, we have: 
\[ \bar M_{A_n}(f) 
= \vert {A_n} \vert^{-1} \sum_{u\in {A_n}} f(X_u) \overset{L^2(\nu)}{\underset{n\to\infty}\longrightarrow}  
c_f 
. \]
\end{theorem}

In particular, when the transition kernel $Q$ is ergodic and $f$ is a continuous bounded function,
then remind that Assumption~\ref{assump:Guyon_simpler} is satisfied and $c_f = \scalProd{\mu}{f}$,
where $\mu$ is the unique invariant measure of $Q$.
If furthermore $Q$ is uniformly ergodic, then 
Assumption~\ref{assump:Q_stronger_ergodic}-\ref{assump_uniformly_ergodic} holds.
Hence, Theorem~\ref{thm:Ergodic_theorem_on_arbitrary_trees} implies Theorem~\ref{thm:intro:Ergodic_theorem_on_arbitrary_trees}.

\begin{proof}
Up to replacing $f$ by $f-c_f$, assume that $c_f=0$.
For all $n\in\N$, we have:
\begin{equation}\label{eq_Esp_M_A_f}
\Esp\bigl[ \bar M_{A_n}(f)^2 \bigr] 
= \vert A_n \vert^{-2} \sum_{u,v\in A_n}  \Esp[ f(X_u) f(X_v) ].
\end{equation}
Remark that for $u,v\in\Tthm$, we have:
\begin{equation}\label{eq_Esp_fXu_fXv}
\Esp[ f(X_u) f(X_v) ]
= \nu Q^{h(u\land v)} \left( Q^{d(u\land v,u)} f \times Q^{d(u\land v,v)} f \right) .
\end{equation}

Set $C = \sup_{n\in\N} \nu Q^n f^2 < \infty$ which is finite
by Assumption~\ref{assump:Guyon_simpler}-\ref{assumpGuyonSimpler1}.
Hence, using Cauchy-Schwarz and Jensen's inequalities,
for all $k,\ell,m\in\N$, we have:
\begin{align}
\vert \nu Q^k ( Q^\ell f \times Q^m f ) \vert
& \leq \Bigl( \nu Q^{k} (Q^{\min(\ell,m)} f)^2 \times \nu Q^{k} (Q^{\max(\ell,m)} f)^2
	\Bigr)^{1/2} 	\nonumber\\
& \leq \Bigl( \nu Q^{k+\min(\ell,m)} f^2 \times \nu Q^{k} (Q^{\max(\ell,m)} f)^2
	\Bigr)^{1/2} 	\nonumber\\
& \leq \sqrt{ C } \times \sqrt{ \nu Q^k (Q^{\max(\ell,m)} f)^2 } .
	 \label{eq_maj_Qmax_ab_f}
\end{align}

Define the distance $\tilde d$ on $\Tthm$ as:
\begin{equation*}
\tilde{d}(u,v) = \max(d(u,u\land v), d(v,u\land v))= \max(h(u),h(v))-h(u\land v) .
\end{equation*}
Remark that we have $d/2 \leq \tilde d \leq d$, thus Assumption~\ref{assump:An_ergodic_theorem}
is equivalent to: for all $k\in\N$, we have $\lim_{n\to\infty} \Prb(\tilde d(U_n, V_n) \leq k) = 0$.
Then, as a consequence of \eqref{eq_Esp_M_A_f}, \eqref{eq_Esp_fXu_fXv} and \eqref{eq_maj_Qmax_ab_f},
we get:
\begin{align*}
\Esp\left[ \bar M_{A_n}(f)^2 \right] 
& \leq \sqrt{C}\times  \vert A_n \vert^{-2} 
	\sum_{u,v\in A_n} \left( \nu Q^{h(u\land v)} \left( Q^{\tilde{d}(u,v)} f\right)^2 \right)^{1/2} \\
& \leq \sqrt{C} \times \left(
	\vert A_n \vert^{-2}
		\sum_{u,v\in A_n} \nu Q^{h(u\land v)} \left( Q^{\tilde{d}(u,v)} f \right)^2
	\right)^{1/2},
\end{align*}
where we used Jensen's inequality in the last inequality.
Hence, to conclude the proof it is enough to prove that the following holds:
\begin{equation}\label{eq_thm_conv_technical}
\Esp\left[ \nu Q^{h(U_n \land V_n)} \left(Q^{\tilde{d}(U_n,V_n)} f\right)^2 \right]
= \vert A_n \vert^{-2} \sum_{u,v\in A_n} 
	\nu Q^{h(u\land v)} \left(Q^{\tilde{d}(u,v)} f\right)^2
\underset{n\to\infty}{\longrightarrow} 0.
\end{equation}
As Assumption~\ref{assump:ancestor_tight} (resp. Assumption~\ref{assump:Q_stronger_ergodic}) holds,
using Lemma~\ref{lemma:assump_2_implies_conv_tech} (resp. Lemma~\ref{lemma:assump_ergo_implies_technical}) below,
we get that \eqref{eq_thm_conv_technical} holds,
which concludes the proof.
\end{proof}

Remind that from Assumption~\ref{assump:Guyon_simpler}, we know that $\lim_{n\to\infty} \nu Q^k (Q^n f)^2 = 0$ for all $k\in\N$,
and that \eqref{eq_thm_conv_technical} adds some uniformity in $k$.
In the next lemma, we check that the ancestral Assumption~\ref{assump:ancestor_tight},
which allows for small values $h(U_n\land V_n)$ with high probability,
implies \eqref{eq_thm_conv_technical}.

\begin{lemma}
	\label{lemma:assump_2_implies_conv_tech}
Let $(A_n)_{n\in\N}$ be a sequence of finite subsets of $\Tthm$ that satisfies 
	Assumption~\ref{assump:An_ergodic_theorem}.
Let $f$ be a function on $\SpaceX$ that satisfies Assumption~\ref{assump:Guyon_simpler}.
Then, Assumption~\ref{assump:ancestor_tight} implies \eqref{eq_thm_conv_technical}.
\end{lemma}

\begin{proof}
Without loss of generality, assume that $c_f = 0$.
Set $C = \sup_{j\in\N} \nu Q^j f^2 < \infty$ which is finite
by Assumption~\ref{assump:Guyon_simpler}-\ref{assumpGuyonSimpler1}.
Thus, for $k, m \in \N$, using Jensen's inequality, we get:
\begin{equation}\label{eq_bound_Qklm_f2}
 \nu Q^k ( Q^m f)^2 
\leq  \nu Q^{k+m} f^2
\leq C 
< \infty.
\end{equation}

Let $\eps>0$.
Using Assumption~\ref{assump:ancestor_tight}, there exists $K\in\N$ such that
	 for $n\in\N$ large enough, 
	we have $\Prb(h(U_n \land V_n) > K) < \eps$.
Using Assumption~\ref{assump:Guyon_simpler}-\ref{assumpGuyonSimpler2},
let $M\in\N$ be such that for all $m\geq M$ and for all $k\leq K$, 
	we have $\nu Q^k (Q^m f)^2 < \eps$.
Using Assumption~\ref{assump:An_ergodic_theorem}, for $n$ large enough 
	we have $\Prb( \tilde d(U_n, V_n) < M) < \eps$.
Hence, using \eqref{eq_bound_Qklm_f2}, for $n$ large enough we get:
\begin{equation*}
\Esp\left[ \nu Q^{h(U_n \land V_n)} \left(Q^{\tilde{d}(U_n,V_n)} f\right)^2 \right]
\leq 2 C \eps 
	+  \max_{k\leq K} \sup_{m\geq M} \nu Q^k (Q^m f)^2
<  (1 + 2C) \eps.
\end{equation*}
This being true for all $\eps>0$,
we get that \eqref{eq_thm_conv_technical} holds, which concludes the proof.
\end{proof}

The following lemma states that the stronger ergodic Assumption~\ref{assump:Q_stronger_ergodic}
implies \eqref{eq_thm_conv_technical};
its proof is similar to the proof of Lemma~\ref{lemma:assump_2_implies_conv_tech} and is left to the reader.

\begin{lemma}\label{lemma:assump_ergo_implies_technical}
Let $(A_n)_{n\in\N}$ be a sequence of finite subsets of $\Tthm$ that satisfies 
	Assumption~\ref{assump:An_ergodic_theorem}.
Let $f$ be a function on $\SpaceX$ that satisfies Assumption~\ref{assump:Guyon_simpler}.
Then, Assumption~\ref{assump:Q_stronger_ergodic} implies \eqref{eq_thm_conv_technical}.
\end{lemma}

\section{Examples satisfying Assumptions~\ref{assump:An_ergodic_theorem} and \ref{assump:ancestor_tight}}
	\label{section:sufficient_conditions}

We now give common examples of trees for which 
Assumptions~\ref{assump:An_ergodic_theorem} and \ref{assump:ancestor_tight} are satisfied.

We denote by $\T$ an arbitrary infinite tree rooted at some vertex $\rooot$.
In this section, all the trees we consider are locally-finite (\ie all nodes have finite degree).
For $n\in\N$, we denote by $\G_n$ the $n$-th generation of $\T$,
	that is the set of vertices at distance $n$ from the root,
and we denote by $\T_n = \cup_{k=0}^n \G_k$ the tree up to generation $n$.
For a vertex $u\in\T$,
we denote by $\T(u)$ the subtree of $\T$ rooted at $u$
		and composed of all descendants of $u$.

\subsection{Some simple deterministic trees}

Firstly, Assumption~\ref{assump:An_ergodic_theorem} is always satisfied
	on trees with bounded vertex degrees,
	in particular for Cayley and Bethe trees
(that is trees where each non-leaf vertex has out-degree $D$ for some $D\geq 1$;
	except for the root of a  Bethe tree which has out-degree $D+1$).

\begin{lemma}[Bounded degree trees]
	\label{lemma:assumpAr2_on_bounded_degree_trees}
Let $D\geq 2$, and let $\T$ be the infinite rooted complete $D$-ary tree 
(that is the tree where each vertex has $D$ children).
Let $(A_n)_{n\in\N}$ be a sequence of finite (non-empty) subsets of $\T$
	such that $\lim_{n\to\infty} \vert A_n \vert = \infty$.
Then, the sequence $(A_n)_{n\in\N}$ satisfies Assumption~\ref{assump:An_ergodic_theorem}.
\end{lemma}

\begin{proof}
Let $k\in\N$. For every vertex $u\in\T$, the ball $B_{\T}(u,k)$ of radius $k$ and center $u$
	has cardinal upper bounded by $c_k = \sum_{j=0}^k (D+1)^j$.
Hence, we have:
\begin{equation*}
\inv{\vert A_n\vert^{2}} \sum_{u,v\in A_n} \ind_{\{ \dgr(u,v)\leq k \}}
= \inv{\vert A_n\vert^{2}} \sum_{u\in A_n} \vert B_{\T}(u,k) \vert
	\leq \inv{\vert A_n\vert} c_k \cdot
\end{equation*}
This implies that Assumption~\ref{assump:An_ergodic_theorem} is satisfied.
\end{proof}

The following counter-example shows that Assumption~\ref{assump:ancestor_tight}
is not always satisfied on a bounded degree tree.

\begin{example}
Let $\T$ be the infinite tree where each vertex has out-degree $D\geq 2$.
Set $A_n = \T(u_n) \cap \G_{2n}$ the $n$-th descendants of $u_n$,
where $u_n = 1 \cdots 1$ ($n$ times) is the left-most $n$-th descendant of the root.
Then, we have $\vert A_n \vert = D^n \to_{n\to\infty} \infty$,
and by Lemma~\ref{lemma:assumpAr2_on_bounded_degree_trees},
the sequence $(A_n)_{n\in\N}$ satisfies Assumption~\ref{assump:An_ergodic_theorem}.
However, we have $\Prb(h(U_n \land V_n) \geq n) = 1$ for all $n\in\N$,
which implies that the sequence $(A_n)_{n\in\N}$ does not satisfy Assumption~\ref{assump:ancestor_tight}.
\end{example}

When with high probability, the vertices in $A_n$ are far from the root (\eg when $A_n =\G_n$)
and have their common ancestor close from the root,
then Assumption~\ref{assump:An_ergodic_theorem} is satisfied.

\begin{lemma}[$A_n$ far from the root]
	\label{lemma:assump_An_Gn_2}
Let $(A_n)_{n\in\N}$ be a sequence of finite subsets of $\Tthm$.
For every $n\in\N$, let $U_n$ and $V_n$ be independent and 
		uniformly sampled elements of $A_n$.
Assume that for every $k\in\N$, $\lim_{n\to\infty} \Prb(h(U_n) \leq k) = 0$
	and that Assumption~\ref{assump:ancestor_tight} holds.
Then, the sequence $(A_n)_{n\in\N}$ satisfies Assumption~\ref{assump:An_ergodic_theorem}.
\end{lemma}

\begin{proof}
Remind that $\dgr(U_n,V_n) = h(U_n) + h(V_n) - 2 h(U_n \land V_n)$.
Thus, for $k\in\N$ we have:
\begin{align*}
\Prb(d(U_n,V_n)\leq 2k)
& \leq 2\, \Prb( h(U_n) - h(U_n \land V_n) \leq k) \\
& \leq 2\, \Prb( h(U_n \land V_n) > k)
	+ 2\, \Prb( h(U_n) \leq 2k ) ,
\end{align*}
where both terms in the upper bound go to zero as $n\to\infty$.
Hence, the sequence $(A_n)_{n\in\N}$ satisfies Assumption~\ref{assump:An_ergodic_theorem}.
\end{proof}

We say a tree $T$ is spherically symmetric (sometimes also called a generalized Bethe tree)
	if for all $n\in\N$, every vertex of height $n$ in $T$ has the the same out-degree $D_n$.
When we choose $A_n = \G_n$ the $n$-th generation for all $n\in\N$,
Assumptions~\ref{assump:An_ergodic_theorem} and~\ref{assump:ancestor_tight} 
are always satisfied on spherically symmetric trees.

\begin{lemma}[Spherically symmetric trees]
	\label{lemma:sperically_symmetric_trees}
Let $\T$ be an infinite spherically symmetric tree such that $\lim_{n\to\infty} \vert \G_n\vert = \infty$.
Then, the sequence $(\G_n)_{n\in\N}$ satisfies Assumptions~\ref{assump:An_ergodic_theorem}
	and~\ref{assump:ancestor_tight}.
\end{lemma}

\begin{proof}
Thanks to Lemma~\ref{lemma:assump_An_Gn_2}, we only need to prove that 
	Assumption~\ref{assump:ancestor_tight} is true.
For all $n\in\N$, denote by $D_n$ the out-degree for all vertices of height $n$.
As $\lim_{n\to\infty} \vert \G_n\vert = \infty$, we have that $D_n>1$ for infinitely many values of $n$.
Let $U_{n}$ and $V_{n}$ be independent random vertices uniformly distributed over $\G_n$.
Using the Ulam-Harris-Neveu tree notation, 
write $U_{n} = U_{(1)} \cdots U_{(n)}$ and $V_{n} = V_{(1)} \cdots V_{(n)}$,
where the random variables $U_{(1)}, \dots, U_{(n)}$, $V_{(1)}, \dots, V_{(n)}$ are independent
	with $U_{(i)}$ and $V_{(i)}$ uniformly distributed over the set $\{1, \dots, D_{i-1} \}$.
Thus, for all $k\in\N$ and $n\geq k$, as $\dgr(U_{n},V_{n}) = 2n - 2 h(U_{n}\land V_{n})$, we have:
\begin{equation*}  
\Prb\Bigl( h(U_{n} \land V_{n}) \geq k \Bigr)
= \Prb\Bigl( U_{(i)} = V_{(i)}, \forall i\in \{ 1, \dots, k\} \Bigr)
= \prod_{i=0}^{k-1} \inv{D_i},
\end{equation*}
where the right hand side goes to $0$ as $k\to\infty$.
This implies that Assumption~\ref{assump:ancestor_tight} is satisfied,
and thus concludes the proof.
\end{proof}

\subsection{Super-critical Bienaymé-Galton-Watson trees}
	\label{section:GW_trees}

To apply the ergodic Theorem~\ref{thm:Ergodic_theorem_on_arbitrary_trees}
to a random sequence $(A_n)_{n\in\N}$ of subsets of $\Tinf$
(independent of the Markov process indexed by $\Tinf$),
we  need to verify that \as the sequence $(A_n)_{n\in\N}$
	satisfy Assumption~\ref{assump:An_ergodic_theorem}
	(and possibly Assumption~\ref{assump:ancestor_tight}).
Note that in this case, 
	Assumptions~\ref{assump:An_ergodic_theorem} and~\ref{assump:ancestor_tight}
	should be considered conditionally on $A_n$,
	in particular the random vertices $U_n$ and $V_n$ are independent and uniformly distributed over $A_n$
	conditionally on $A_n$.

In this subsection, we consider the case where $\T$
is a super-critical Bienaymé-Galton-Watson tree whose offspring distribution $\cP$ on $\N$
	has finite mean $m>1$ and finite second moment,
	and is conditioned on non-extinction.
The following lemma states that \as both the sequences $(\G_n)_{n\in\N}$ and $(\T_n)_{n\in\N}$ 
	satisfy Assumptions~\ref{assump:An_ergodic_theorem} and \ref{assump:ancestor_tight}.

\begin{lemma}[Super-critical Bienaymé-Galton-Watson trees]
	\label{lemma:GW_Gn_dist_distrib}
Let $\T$ be a super-critical Bienaymé-Galton-Watson tree whose offspring distribution has mean $m>1$ and finite second moment
	and is conditioned on non-extinction.
\begin{enumerate}[label=(\roman*)]
\item\label{item:GW_An_dist_distrib}
Let $(\ell_n)_{n\in\N}$ be a sequence of integers such that $0\leq \ell_n \leq n$ for all $n\in\N$.
For every $n\in\N$,
let $A_n = \cup_{k=(n-\ell_n)_+}^n \G_{k}$
be the subset composed of the last $\ell_n +1$ generations of $\T_n$.
Then, the sequence $(A_n)_{n\in\N}$ \as satisfies Assumptions~\ref{assump:An_ergodic_theorem} and \ref{assump:ancestor_tight}.
\item\label{item:GW_Gn_Tn_dist_distrib}
The sequences $(\G_n)_{n\in\N}$ and $(\T_n)_{n\in\N}$ \as satisfies 
Assumptions~\ref{assump:An_ergodic_theorem} and \ref{assump:ancestor_tight}.
\end{enumerate}
\end{lemma}

In the case of a critical Bienaymé-Galton-Watson tree, it is not possible to condition on non-extinction,
but taking $A_n = \G_n$ and conditioning on non-extinction at time $n$,
\cite[Theorem~2.1]{athreyaCoalescenceCriticalSubcritical2012} suggests that
Assumption~\ref{assump:An_ergodic_theorem} should be satisfied but not Assumption~\ref{assump:ancestor_tight}.

The proof of Lemma~\ref{lemma:GW_Gn_dist_distrib}-\ref{item:GW_An_dist_distrib}
relies on the case where the sequence $(\ell_n)_{n\in\N}$ is bounded.
In this case,
we prove a stronger version of \cite[Theorem~2]{athreyaCoalescenceRecentRapidly2012}
where individuals can have zero child 
and the two random individuals are taken from the last $\ell_n+1$ generations instead of only the last generation.

Before proving Lemma~\ref{lemma:GW_Gn_dist_distrib},
	we need to prove the following lemma stating that the last generations
	carry most of the weight of $\T_n$.

\begin{lemma}[Last generations carry all the weight]
	\label{lemma:last_gens_Tn}
Let $\T$ be a super-critical Bienaymé-Galton-Watson tree whose offspring distribution has mean $m>1$,
	and is conditioned on non-extinction.
We have:
\begin{equation}\label{eq_ratio_Gn_Tn}
\forall \ell\in\N,\qquad
\text{\as}\quad
\lim_{n\to\infty} \frac{\vert \cup_{k=(n-\ell)_+}^n \G_{k} \vert}{\vert \T_n\vert} = 1 - \inv{m^{\ell+1}}
	\cdot
\end{equation}
\end{lemma}

\begin{proof}
For all $n\in\N$, we write $Z_n = \vert \G_n \vert $.
Using \cite[Theorem~I.C.12.3]{athreyaBranchingProcesses1972},
as $\T$ is a super-critical Bienaymé-Galton-Watson tree whose offspring distribution $\cP$
	has a finite mean $m>1$,
	and is conditioned on non-extinction,
we know that 
	there exists a sequence of positive constants $(C_n)_{n\in\N}$
	with $\lim_{n\to\infty} C_n = \infty$ and $\lim_{n\to\infty} C_{n+1} / C_n = m$,
	and such that
	\as $\lim_{n\to\infty} C_n^{-1} Z_n = W$ 
	where $W$ is a random variable on $\R_+$,
	and where the event $\{ W = 0 \} \supset \{ \exists n,\, Z_n = 0 \}$ has zero probability.
Hence, we get that \as $C_n^{-1} \vert \cup_{k=0}^\ell \G_{n-k} \vert 
= C_n^{-1} \sum_{k=0}^\ell Z_{n-k}$
converges to $W \sum_{k=0}^\ell m^{-k}$ as $n$ goes to infinity.
Writing $S_n = \sum_{k=0}^n C_k$, as we know that $C_n \sim m C_{n-1}$,
we have that $S_{n+1} = S_n + C_{n+1} = m S_n + o(S_n)$,
and thus we get that $C_n \sim C_{n+1} / m \sim \frac{m-1}{m} S_n$.
As $\lim_{n\to\infty} C_n^{-1} \sum_{k=0}^n C_{k} = \frac{m}{m-1}$, 
using Stolz-Cesàro theorem, 
	we get that \as $C_n^{-1} \vert \T_n \vert = C_n^{-1} \sum_{k=0}^n Z_{k}$
	converges to $W \frac{m}{m-1}$ as $n$ goes to infinity.
Consequently, we get \eqref{eq_ratio_Gn_Tn}.
\end{proof}

\begin{proof}[Proof of Lemma~\ref{lemma:GW_Gn_dist_distrib}]
Point~\ref{item:GW_Gn_Tn_dist_distrib} is an immediate consequence of 
Point~\ref{item:GW_An_dist_distrib} taking $\ell_n = 0$ or $n$.

We now prove Point~\ref{item:GW_An_dist_distrib}.
Using Lemma~\ref{lemma:assump_An_Gn_2} with \eqref{eq_ratio_Gn_Tn}, it is enough to prove that
the sequence $(A_n)_{n\in\N}$ \as satisfies Assumption~\ref{assump:ancestor_tight}.
Let $U_n$ and $V_n$ be independent and uniformly sampled elements of $A_n$.
For $k\in\N$, remark that $h(U_n \land V_n) \geq k$ is equivalent to the existence of $u\in\G_k$
	such that $U_n,V_n\in \T(u)$.
Note that Assumption~\ref{assump:ancestor_tight} for $(A_n)_{n\in\N}$ in the random tree $\T$ can be reformulated as:
\begin{align}
\lim_{k\to\infty} \limsup_{n\to\infty}\, \Prb\Bigl( h(U_n \land V_n) \geq k \, \Bigm\vert \,  \T \Bigr) 
& = \lim_{k\to\infty} \limsup_{n\to\infty}\, \Prb\Bigl( \exists u \in \G_{k},\,  U_n,  V_n \in \T(u) \, \Bigm\vert \,  \T \Bigr) 
	\nonumber \\
& = 0 . \label{eq_cond_assump_2}
\end{align}
We divide the rest of the proof into two cases:
we first consider the case when the sequence $(\ell_n)_{n\in\N}$ is bounded,
and then the general case.
\medskip

\textbf{Case 1: the sequence $(\ell_n)_{n\in\N}$ is bounded.}
Following \cite[Section I.D.12]{athreyaBranchingProcesses1972}
(and using the survivor/extinct vertices denomination from \cite[Section 2.2.3]{IntroGWTrees}),
the vertices of the super-critical Bienaymé-Galton-Watson tree $\T$ conditioned on non-extinction 
can be partitioned into two categories:
	\emph{survivor} vertices, whose descendants do not suffer extinction,
	and \emph{extinct} vertices, whose descendants eventually become extinct.
The root of $\T$ is a survivor vertex due to the conditioning on non-extinction.
Denote by $\T^s$ the subtree of $\T$ composed of survivor vertices.
From \cite[Theorem I.D.12.1]{athreyaBranchingProcesses1972}, the tree $\T^s$ is distributed 
	as a super-critical Bienaymé-Galton-Watson tree
	whose offspring distribution $\cP_0$ on $\N^*$ (hence no extinction)
		has the same mean $m>1$ as $\cP$ and has finite second moment.
Also, from \cite[Theorem I.D.12.3]{athreyaBranchingProcesses1972}, we know that if $u\in\T$ is an extinct vertex,
	then its descendant subtree $\T(u)$ is distributed as a sub-critical Bienaymé-Galton-Watson tree
	whose offspring distribution $\cP_1$ is explicitly known
		and has finite second moment.

For all $n> \ell$ and $u\in\T$, define $Z_{n,\ell,u} = \vert (\T_n \setminus \T_{n-\ell-1}) \cap \T(u) \vert$.
Fix some $k\in\N$.
Conditionally on $\T$, the unique ancestor of $U_n$ in $\G_{k}$ is $u\in \G_{k}$
	with probability $Z_{n,\ell_n,u} / \sum_{v\in\G_{k}} Z_{n,\ell_n,v}$.
Define $Z_n = \vert \G_n \vert$, $\G_n^s = \G_n\cap\T^s$
	and $Z_n^s = \vert \G_n^s \vert$.
Then, we have:
\begin{equation*}
\Prb\Bigl( \exists u \in \G_{k},\,  U_n,  V_n \in \T(u) \, \Bigm\vert \,  \T \Bigr) 
= \frac{ \sum_{u\in\G_k} Z_{n,\ell_n,u}^2 }{ \left( \sum_{u\in\G_{k}} Z_{n,\ell_n,u} \right)^2 }	
=  \frac{ \sum_{u\in\G_k^s} Z_{n,\ell_n,u}^2 }{ \left( \sum_{u\in\G_{k}^s} Z_{n,\ell_n,u} \right)^2 }	
, 		\label{eq_upper_bound_proba_Gn_ext}
\end{equation*}
where the last equality holds \as for $n$ large enough
as for any extinct vertex $u\in\G_k\setminus \G_k^s$, we know that \as $Z_{n,\ell_n,u} = 0$ for $n$ large enough
(remind that the sequence $(\ell_n)_{n\in\N}$ is bounded).

From \cite[Theorems I.B.6.1 and I.B.6.2]{athreyaBranchingProcesses1972},
as $\cP$ has finite second moment,
we know that $\lim_{n\to\infty} m^{-n} Z_n = W$ \as and in $L^2$
where $W$ is a random variable in $\R_+^*$ with finite second moment.
(From \cite[Theorem I.B.6.2-(iii)]{athreyaBranchingProcesses1972}, 
	we know that on the non-extinction event, 
	that is when the root is a survivor vertex, \as $W$ is positive.)
Then, remark that for all survivor vertices $u\in\G_k^s$, we have that \as $\lim_{n\to\infty} m^{-(n-k)} Z_{n,0,u} = W_u$,
where conditionally on $\G_k^s$ the random variables $(W_u)_{u\in\G_k^s}$ are independent and distributed as $W$
(remind that the root of $\T$ is also a survivor vertex).
Thus, we get that \as for all $\ell\in\N$ and all $u\in\G_k^s$, 
	$\lim_{n\to\infty} m^{-(n-k)} Z_{n,\ell,u} = (\sum_{j=0}^\ell m^{-j}) W_u$.
As the sequence $(\ell_n)_{n\in\N}$ is bounded, this implies that \as for all vertices $u\in\G_k^s$, 
	$\lim_{n\to\infty} m^{-(n-k)} (\sum_{j=0}^{\ell_n} m^{-j})^{-1}\, Z_{n,\ell_n,u} = W_u$.
Hence, we get:
\begin{equation*}
\text{\as} \qquad \frac{ \sum_{u\in\G_k^s} Z_{n,\ell,u}^2 }{ \left( \sum_{u\in\G_{k}^s} Z_{n,\ell,u} \right)^2 }
\underset{n\to\infty}{\longrightarrow}
\frac{ \sum_{u\in\G_k^s} W_u^2 }{ \left( \sum_{u\in\G_{k}^s} W_u \right)^2 }
\cdot
\end{equation*}
Combining what we got so far, we get:
\begin{equation}\label{eq_lim_proba_GW_ancestor}
\text{\as} \qquad
 \lim_{n\to\infty}\, \Prb\Bigl( h(U_n \land V_n) \geq k \, \Bigm\vert \,  \T \Bigr) 
=   \frac{ \sum_{u\in\G_k^s} W_u^2 }{ \left( \sum_{u\in\G_{k}^s} W_u \right)^2 }
\cdot
\end{equation}
As the left hand side in \eqref{eq_lim_proba_GW_ancestor} is a non-increasing function of $k$,
then so is the right hand side in \eqref{eq_lim_proba_GW_ancestor}.
Thus, taking the limit when $k\to\infty$ and then taking the expectation, we get:
\begin{align}
\Esp\left[ \lim_{k\to\infty} \lim_{n\to\infty}\, \Prb\Bigl( h(U_n \land V_n) \geq k \, \Bigm\vert \,  \T \Bigr)  \right]
& = \Esp\left[ \lim_{k\to\infty}  \frac{ \sum_{u\in\G_k^s} W_u^2 }{ \left( \sum_{u\in\G_{k}^s} W_u \right)^2 }  \right] \nonumber\\
& = \lim_{k\to\infty} \Esp\left[  \frac{ \sum_{u\in\G_k^s} W_u^2 }{ \left( \sum_{u\in\G_{k}^s} W_u \right)^2 }  \right]
, \label{eq_upper_bound_R_Zk}
\end{align}
where we used the dominated convergence theorem in the last inequality.

Let $(W_n)_{n\in\N}$ be a sequence of independent random variables distributed as $W$.
For all $n\in\N$, define the random variable $R(n) = { \sum_{i=1}^n W_i^2 } / { \left( \sum_{i=1}^n W_i \right)^2 }$.
Using the strong law of large numbers, we get that the numerator of $R(n)$ is \as equivalent to $n \Esp[W^2]$,
 and the denominator of $R(n)$ is \as equivalent to $n^2 \Esp[W]^2$.
This implies that $R(n)$ is \as equivalent to $n^{-1} \Esp[W^2] / \Esp[W]^2$, and thus \as $\lim_{n\to\infty} R(n) = 0$.
Remark that the expectation in the last line of \eqref{eq_upper_bound_R_Zk} can be written as $\Esp[ R(Z_k^s) ]$.
As we know that \as $\lim_{k\to\infty} Z_k^s = \infty$, we get that \as $\lim_{k\to\infty} R(Z_k^s) = 0$.
Hence, using the dominated convergence theorem (with domination by $1$),
we get that $\lim_{n\to\infty} \Esp[ R(Z_k^s) ] = 0$,
which implies that the left hand side in \eqref{eq_upper_bound_R_Zk} is also null.
As a consequence, we get that \as \eqref{eq_cond_assump_2} holds,
which implies that \as the sequence $(A_n)_{n\in\N}$ satisfies Assumption~\ref{assump:ancestor_tight}.
This concludes the proof of the first case.
\medskip

\textbf{Case 2: general case.}
Let $\eps>0$, and let $\ell\in\N$ be such that $m^{-\ell-1} < \eps$.
Then, using Lemma~\ref{lemma:last_gens_Tn}, we get that
\as $\Prb( h(U_n) < n-\ell \, \vert \, \T ) < \eps$ for $n$ large enough.
Thus, \as for $n$ large enough, for all $k\in\N$, we have:
\begin{align}
\Prb( h(U_n\land V_n) \geq k \, \vert \, \T )
& \leq \Prb( h(U_n) < n-\ell \text{ or } h(V_n) < n-\ell \, \vert \, \T ) 
	\nonumber\\
& \qquad	+ \Prb( h(U_n) \geq n-\ell \text{ and } h(V_n) \geq n-\ell \, \vert \, \T ) 
	\nonumber\\
& \qquad \qquad \times \Prb( h(U_n\land V_n) \geq k  \,\vert\, \min(h(U_n), h(V_n)) \geq n-\ell,\, \T ) 
	\nonumber\\
& \leq 2 \eps +  \Prb( h(\bar U_n \land \bar V_n) \geq k \, \vert \, \T ),
		 \label{eq_major_proba_Gn_bis}
\end{align}
where $\bar U_n$ and $\bar V_n$ are independent and 
	uniformly distributed over $\cup_{k=(n-{\min(\ell_n, \ell)})_+}^n \G_{k}$.
By the first case, we have that
	$\lim_{k\to\infty} \limsup_{n\to\infty} \Prb( h(\bar U_n \land \bar V_n) \geq k \, \vert \, \T ) = 0$.
Thus, using \eqref{eq_major_proba_Gn_bis}, we get that
$\lim_{k\to\infty} \limsup_{n\to\infty} \Prb( h(U_n\land V_n) \geq k \, \vert \, \T ) \leq 2 \eps$.
This being true for all $\eps > 0$, we get that \as the sequence $(A_n)_{n\in\N}$ 
	satisfies Assumption~\ref{assump:ancestor_tight}, which concludes the proof.
\end{proof}

\section{Dependence of the variance on the shape of the tree}
	\label{section:variance}

In this section, we briefly discuss the variance of the empirical average estimator $\bar M_A(f)$
(which estimates $\scalProd{\mu}{f}$),
that is $\Esp_{\mu}\bigl[ \vert A\vert^{-1} M_A(f)^2\bigr]$ for some Borel function $f$ on $\SpaceX$,
and its dependence on the geometry of the averaging set $A\subset \T$.
We consider the case where the transition kernel $Q$ induces a self-adjoint compact operator on $L^2(\mu)$,
where $\mu$ is the unique invariant measure of $Q$
(note that self-adjoint is equivalent to $(\mu,Q)$ being reversible).
This is in particular the case when the state space $\SpaceX$ is finite and $(\mu,Q)$ is reversible,
or when the operator induced by $Q$ on $L^2(\mu)$ is a symmetric Hilbert-Schmidt operator.
 
We now prove Proposition~\ref{prop:intro:lemma_var_connected_sets},
which is a non-asymptotic result
that states that among subtrees $T\subset \Tinf$ of a given finite size, 
the line graph tree (\ie the Markov chain) is the one minimizing the variance of the empirical average estimator.

\begin{proof}[Proof of Proposition~\ref{prop:intro:lemma_var_connected_sets}]
Let $f\in L^2(\mu)$ be some function.
As $Q$ induces a self-adjoint compact operator on $L^2(\mu)$,
using \cite[Theorem 12.29-(d) and Theorem 12.30]{rudinFunctionalAnalysis1996}
(remind that a self-adjoint operator is normal),
the spectrum of $Q$ is composed of a (at most) countable number of eigenvalues $(\alpha_k)_{k\in\N}$, 
and the function $f\in L^2(\mu)$ has a unique expansion $\sum_{k\in\N} f_k$
where we have $Q f_k = \alpha_k f_k$ for all $k\in\N$, and $\scalProd{f_k}{f_\ell}_{L^2(\mu)} =0$ for $k\neq \ell$.
As $Q$ is a Markov kernel, we have $\alpha_k \in [-1,1]$ for all $k\in\N$
(indeed, using Jensen's inequality, we have $\alpha_k^2 \scalProd{\mu}{f_k^2} = \scalProd{\mu}{(Q f_k)^2} \leq \scalProd{\mu}{Q (f_k^2)} = \scalProd{\mu}{f_k^2}$).
Remark that for all $n,m\in\N$, we have:
\begin{equation}\label{eq_Qn_fi}
\scalProd{\mu}{Q^n f\times Q^m f}
= \sum_{i,j\in\N} \scalProd{Q^n f_i}{Q^m f_j}_{L^2(\mu)}
= \sum_{i\in\N}  \alpha_i^{n+m} \scalProd{\mu}{f_i^2} .
\end{equation}
Using \eqref{eq_Esp_fXu_fXv} with $\nu=\mu$,  
we get that $\Esp_\mu[ f(X_u) f(X_v) ] = \sum_{k\in\N}  \alpha_k^{d(u,v)} \scalProd{\mu}{f_k^2}$.
In particular, we get that
$\Esp_\mu[ M_A(f)^2 ] = \sum_{k\in\N}  \Esp_\mu[ M_A(f_k)^2 ]$,
and thus it is sufficient to prove the result when $f=f_k$ for some $k\in\N$.
(Remark that there exists $k\in\N$ such that $f_k \neq 0$ and $\alpha_k \not\in \{ -1,0,1\}$
if and only if $f \not\in \Ker(Q) \oplus \Ker(Q-I) \oplus \Ker(Q+I)$.) 

Hence, let $k\in\N$ be fixed, 
and in the rest of the proof, we will write $f = f_k$ and $\alpha = \alpha_k\in [-1,1]$.
Thus, the variance of the empirical average estimator can be written as:
\begin{equation*}
 \vert A \vert^{-1}\ \Esp_\mu[ M_A(f)^2 ]
	= \vert A \vert^{-1} \scalProd{\mu}{f^2}\, H_A(\alpha) ,
	\quad\text{with}\quad
H_A(\alpha) = \sum_{u,v \in A} \alpha^{d(u,v)}  , 
\end{equation*}
which involves the Hosoya-Wiener polynomial $H_A(\alpha)$ of the subset $A$.
Note that variance minimization is equivalent to minimization of $H_A(\alpha)$.
Also note that we may consider unrooted trees $A=T$, as the definition of $H_T(\alpha)$ is invariant by rerooting the tree.
Hence, the proof of the proposition is complete by applying
Lemma~\ref{lemma_minimization_Hosoya_polynomial}.
\end{proof}

We now prove Lemma~\ref{lemma_minimization_Hosoya_polynomial} 
which states that the Hosoya-Wiener polynomial
is minimized by the line graph tree among trees of a given size.

%
%

\begin{proof}[Proof of Lemma~\ref{lemma_minimization_Hosoya_polynomial}]
If $\alpha=0$ or $1$, then the value of the Hosoya-Wiener polynomial $H_T$ depends only on the size of $T$,
and thus is the same for every tree $T$ of size $n$.
If $\alpha=-1$, then we have $H_T(-1) = ( \vert B \vert - \vert R \vert )^2$ where $T = B\cup R$ is the bipartite partitioning of vertices in $T$,
that is the value of $H_T(-1)$ is the imbalance between the two bipartite classes of vertices in $T$
($2$-coloring of $T$), and its minimal value is $0$ (resp. $1$) when $n$ is even (resp. odd),
and is achieved by the line graph (but not uniquely for $n\geq 5$, \eg the double-cherry graph in Figure~\ref{fig:double_cherry_graph} is also a minimizer for $n=6$).
We now assume that $\alpha \in (-1,1)\setminus\{0\}$,
and we are going to prove that in this case the line graph is the unique minimizer of $H_T(\alpha)$.
We divide the rest of the proof in two cases depending on the sign of $\alpha$.

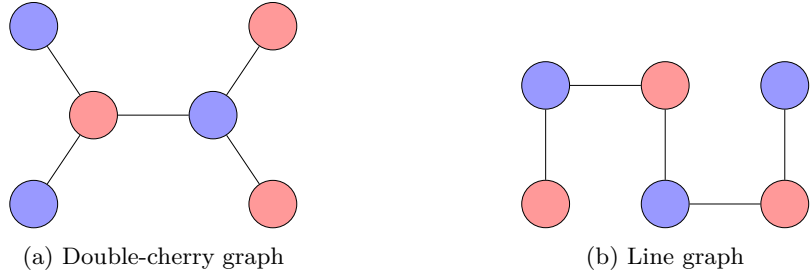
\begin{figure}[t]
     \centering
     \begin{subfigure}[b]{0.45\myFigWidth}
         \centering
\resizebox {0.6\myFigSize} {!} {
\begin{tikzpicture}[
  blueVertex/.style={circle, fill=black!100, draw, minimum size=8mm},
  redVertex/.style={circle, fill=white!40, draw, minimum size=8mm},,
  edge/.style={draw}
]

\node[redVertex] (v) at (0,0) {};
\node[blueVertex] (v1) at (-1,1.5) {};
\node[blueVertex] (v2) at (-1,-1.5) {};
\node[blueVertex] (u) at (2,0) {};
\node[redVertex] (u1) at (3,1.5) {};
\node[redVertex] (u2) at (3,-1.5) {};

\draw[edge] (v) -- (u);
\draw[edge] (u) -- (u1);
\draw[edge] (u) -- (u2);
\draw[edge] (v) -- (v1);
\draw[edge] (v) -- (v2);
\end{tikzpicture}
}

\caption{Double-cherry graph}
\label{fig:double_cherry_graph}
\end{subfigure}
     \begin{subfigure}[b]{0.45\myFigWidth}
         \centering
\resizebox {0.6\myFigSize} {!} {
\begin{tikzpicture}[
  blueVertex/.style={circle, fill=black!100, draw, minimum size=8mm},
  redVertex/.style={circle, fill=white!40, draw, minimum size=8mm},,
  edge/.style={draw}
]

\node[redVertex] (u1) at (0,0) {};
\node[blueVertex] (u2) at (0,2) {};
\node[redVertex] (u3) at (2,2) {};
\node[blueVertex] (u4) at (2,0) {};
\node[redVertex] (u5) at (4,0) {};
\node[blueVertex] (u6) at (4,2) {};

\draw[edge] (u1) -- (u2);
\draw[edge] (u2) -- (u3);
\draw[edge] (u3) -- (u4);
\draw[edge] (u4) -- (u5);
\draw[edge] (u5) -- (u6);
\end{tikzpicture}
}

\caption{Line graph}
\end{subfigure}

\caption{Comparison of the double-cherry graph and the line graph ($n=6$);
both graphs have an exactly balanced bipartite $2$-coloring,
and thus satisfy $H_T(-1)=0$.}
\end{figure}
\medskip

\textbf{Proof for $\alpha\in (0,1)$.}
The case $\alpha\in (0,1)$ was already proved in \cite[Theorem~9]{casablancaDistanceEccentricSequences2019},
but as it is a simple argument, we include it here for the sake of completeness.
Let $n\in\N^*$.
To prove that the line graph minimizes the function $T\mapsto H_T(\alpha)$ among trees of size $n$,
Let $u_1$ be a leaf of $T$, and consider the tree $T$ to be rooted at $u_1$.
Let $\ell\in\N^*$ be the first generation with size larger than $1$
(which exists as $T$ is not the line graph),
and for $i<\ell$, denote by $u_{i+1}$ the only vertex in the $i$-th generation.
Let $v$ be one of the children of the vertex $u_\ell$,
and denote by $T_v$ the connected component of $T\setminus \{ u_\ell \}$ that contains $v$.
We define the tree $T'$ by removing the edge $(u_\ell, v)$ and by adding the edge $(u_1, v)$
(see Figure~\ref{fig:tree_alpha_pos_before},
	where for clarity the other children of $u_\ell$ have been labeled $v_2,\cdots, v_k$).
We now compare the distances $d_{T'}(u,w)$ and $d_{T}(u,w)$ for $u,w\in T$:
	if $u$ and $w$ are both in $T\setminus T_v$ or both in $T_v$,
	then $d_{T'}(u,w) = d_{T}(u,w)$; 
	if $w\in T_v$, then $d_{T'}(u_i, w) = d_{T}(u_{\ell-i+1},w)$ for $1\leq i \leq \ell$;
	and if $u\in T\setminus (T_v \cup \{ u_1, \cdots, u_\ell \})$ and $w\in T_v$,
		then we have: 
\[d_{T'}(u,w) 
	= d_{T'}(u,u_\ell) + d_{T'}(u_\ell,v) +  d_{T'}(w,v) 
	= d_{T}(u,u_\ell) + \ell + d_{T}(w,v)
	>
	d_{T}(u,w) .\]
Hence, we get: 
\begin{equation*}
H_{T}(\alpha) - H_{T'}(\alpha) 
= 2 \sum_{u\in T_v, w\in T\setminus (T_v \cup \{ u_1, \cdots, u_\ell \})} \alpha^{d_T(u,v)} - \alpha^{d_{T'}(u,v)}
> 0,
\end{equation*}
and thus $T$ does not minimize $H_T(\alpha)$.

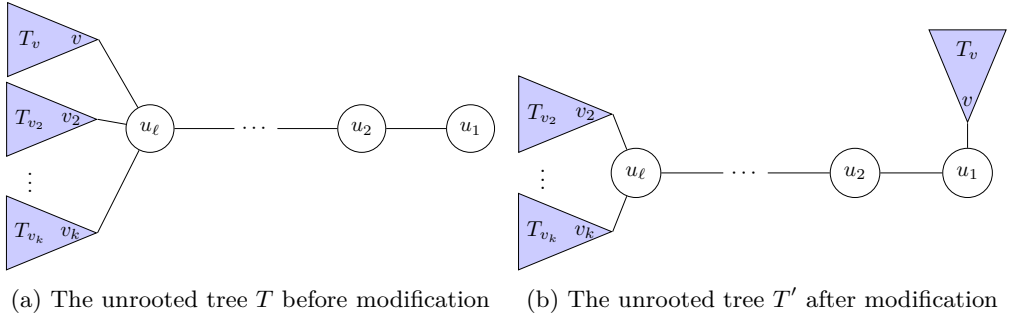
\begin{figure}[t]
     \centering
     \begin{subfigure}[b]{0.45\myFigWidth}
         \centering
\resizebox {1\myFigSize} {!} {
\begin{tikzpicture}[
  vertex/.style={circle, draw, minimum size=8mm},
  tree/.style={isosceles triangle, fill=blue!20, draw, minimum height=1.5cm, minimum width=1cm, anchor=east},
  treesouth/.style={isosceles triangle, fill=blue!20, shape border rotate=-90, draw, minimum height=1.5cm, minimum width=1cm, anchor=south},
  subtree/.style={trapezium, draw, fill=blue!20, shape border rotate=180, trapezium angle=70, minimum height=0.8cm, minimum width=0.4cm, anchor=south},
  edge/.style={draw}
]

\node[vertex] (uell) at (0,0) {$u_\ell$};
\node[, right=1 of uell] (udots) {$\cdots$};
\node[vertex, right=1 of udots] (u2) {$u_2$};
\node[vertex, right=1 of u2] (u1) {$u_1$};

\node[] (vdots) at (-2,-0.8) {$\vdots$};
\node[tree, above=0.1cm of vdots] (v2) {$T_{v_2}$};
\node[tree, above=0.4cm of v2] (v) {$T_{v}$};
\node[tree, below=0.1cm of vdots] (vk) {$T_{v_k}$};

\node[, left=0.15cm of v.apex] (oo_v) {$v$};
\node[, left=0.15cm of v2.apex] (oo_v2) {$v_2$};
\node[, left=0.15cm of vk.apex] (oo_vk) {$v_k$};

\draw[edge] (uell) -- (udots);
\draw[edge] (u2) -- (udots);
\draw[edge] (u2) -- (u1);
\draw[edge] (uell) -- (v2.apex);
\draw[edge] (uell) -- (v.apex);
\draw[edge] (uell) -- (vk.apex);

\end{tikzpicture}
}
\caption{The unrooted tree $T$ before modification}
\label{fig:tree_alpha_pos_after}
\end{subfigure}
     \begin{subfigure}[b]{0.45\myFigWidth}
         \centering
\resizebox {1\myFigSize} {!} {
\begin{tikzpicture}[
  vertex/.style={circle, draw, minimum size=8mm},
  tree/.style={isosceles triangle, fill=blue!20, draw, minimum height=1.5cm, minimum width=1cm, anchor=east},
  treesouth/.style={isosceles triangle, fill=blue!20, shape border rotate=-90, draw, minimum height=1.5cm, minimum width=1cm, anchor=south},
  subtree/.style={trapezium, draw, fill=blue!20, shape border rotate=180, trapezium angle=70, minimum height=0.8cm, minimum width=0.4cm, anchor=south},
  edge/.style={draw}
]

\node[vertex] (uell) at (0,0) {$u_\ell$};
\node[, right=1 of uell] (udots) {$\cdots$};
\node[vertex, right=1 of udots] (u2) {$u_2$};
\node[vertex, right=1 of u2] (u1) {$u_1$};

\node[] (vdots) at (-1.5,0) {$\vdots$};
\node[tree, above=0.1cm of vdots] (v2) {$T_{v_2}$};
\node[treesouth, above=0.4cm of u1] (v) {$T_{v}$};
\node[tree, below=0.1cm of vdots] (vk) {$T_{v_k}$};

\node[, above=0.15cm of v.apex] (oo_v) {$v$};
\node[, left=0.15cm of v2.apex] (oo_v2) {$v_2$};
\node[, left=0.15cm of vk.apex] (oo_vk) {$v_k$};

\draw[edge] (uell) -- (udots);
\draw[edge] (u2) -- (udots);
\draw[edge] (u2) -- (u1);
\draw[edge] (uell) -- (v2.apex);
\draw[edge] (u1) -- (v.apex);
\draw[edge] (uell) -- (vk.apex);

\end{tikzpicture}
}
\caption{The unrooted tree $T'$ after modification}
\label{fig:tree_alpha_pos_before}
\end{subfigure}

\caption{The unrooted trees $T$ and $T'$ for $\alpha\in(0,1)$}
\label{fig:tree_alpha_pos}
\end{figure}
\medskip

\textbf{Proof for $\alpha\in (-1,0)$.}
We prove the statement by recurrence on the size $n$ of the tree.
For $n\in \{1,2,3\}$, there exists only one (unrooted) tree of size $n$, hence the statement is trivial.

Before proving that this property is hereditary, 
first remark that if $T$ is a tree of size $n\geq 2$ and $u,v\in T$ are two vertices connected by an edge,
and we let $T_u$ and $T_v$ be the two rooted subtrees of $T$ obtained by removing the edge $(u,v)$
and rooted at $u$ and $v$ respectively,
then we get:
\begin{align}
H_T(\alpha) 
& = H_{T_u}(\alpha)  + H_{T_v}(\alpha)  + 2 \sum_{u' \in T_u} \sum_{v' \in T_v} \alpha^{d_T(u',v')} \nonumber\\
& = H_{T_u}(\alpha)  + H_{T_v}(\alpha)  + 2\alpha\, C_{T_u}(\alpha) C_{T_v}(\alpha)
	, \label{eq_H_T_rec}
\end{align}
where for convenience we write $C_{T_v}(\alpha) = \left( \sum_{v' \in T_v} \alpha^{d_T(v,v')} \right)$,
and similarly for $T_u$.
We also remark that when $T_v$ is the line graph with $k$ vertices $\{ u_1, \cdots, u_k\}$ rooted at the vertex $v=u_j$, 
then we have $C_{T_v}(\alpha) = \inv{1-\alpha}( 1 + \alpha - \alpha^{j} - \alpha^{k-j+1})$,
which is maximal among rooted copies of $\{ u_1, \cdots, u_k\}$ only for $j=1$ or $k$
(remind that $\alpha\in (-1,0)$).
We denote by $L_k$ the line graph with $k$ vertices rooted at one of its extremities.

Now, consider $n>3$ and assume that for all $k<n$, 
the line graph is the unique minimizer of $T\mapsto H_T(\alpha)$ among trees of size $k$.
Consider an unrooted tree $T$ of size $n$ that is not the line graph.
Let $u$ be a leaf of the tree obtained by removing all the leaves of $T$
(such vertex is sometimes called a non-protected vertex).
In particular, the node $u$ has $2+\ell$ neighbors in $T$ with $\ell\in\N$, which we denote by $u_0, \cdots, u_{\ell+1}$;
and at most one of the neighbors of $u$ is not a leaf, say $v=u_{\ell+1}$.
(All the neighbors of $u$ are leaves if and only if $T$ is the star graph whose center vertex is $u$,
	in which case we still write $v=u_{\ell+1}$.)
Denote by $T_v$ the subtree $T\setminus \{ u, u_0, \cdots, u_\ell \}$ of $T$ rooted at $v$.
We consider three cases.

\begin{figure}[t]
     \centering
     \begin{subfigure}[b]{0.3\myFigWidth}
         \centering
\resizebox {0.8\myFigSize} {!} {
\begin{tikzpicture}[
  vertex/.style={circle, draw, minimum size=8mm},
  subtree/.style={trapezium, draw, fill=blue!20, shape border rotate=270, trapezium angle=60, minimum height=0.8cm, minimum width=0.5cm, anchor=east},
    tree/.style={isosceles triangle, fill=blue!20, draw, minimum height=1.5cm, minimum width=1cm, anchor=east},
  edge/.style={draw}
]

\node[tree] (v) at (0.5,0) {$T_v$};
\node[, left=0.15cm of v.apex] (oo_v) {$v$};
\node[vertex] (u) at (1.5,0) {$u$};
\node[vertex] (u0) at (3,0) {$u_0$};

\draw[edge] (v) -- (u);
\draw[edge] (u) -- (u0);

\end{tikzpicture}
}
\caption{Case 1}
\label{fig:case_1}
     \end{subfigure}
     \hfill
     \begin{subfigure}[b]{0.3\myFigWidth}
    \centering 
\resizebox {0.8\myFigSize} {!} {
\begin{tikzpicture}[
  vertex/.style={circle, draw, minimum size=8mm},
  subtree/.style={trapezium, draw, fill=blue!20, shape border rotate=270, trapezium angle=60, minimum height=0.8cm, minimum width=0.5cm, anchor=east},
  tree/.style={isosceles triangle, fill=blue!20, draw, minimum height=1.5cm, minimum width=1cm, anchor=east},
  edge/.style={draw}
]

\node[tree] (v) at (0.5,0) {$T_v$};
\node[, left=0.15cm of v.apex] (oo_v) {$v$};
\node[vertex] (u) at (1.5,0) {$u$};
\node[vertex] (u0) at (3,0) {$u_0$};
\node[] (dots) at (1.5,-1.5) {$\cdots$};
\node[vertex, left=0.1cm of dots] (u1) {$u_1$};
\node[vertex, right=0.1cm of dots] (ul) {$u_\ell$};

\draw[edge] (v) -- (u);
\draw[edge] (u) -- (u0);
\draw[edge] (u) -- (u1);
\draw[edge] (u) -- (ul);

\end{tikzpicture}
}
\caption{Case 2}
\label{fig:case_2}
     \end{subfigure}
     \hfill
     \begin{subfigure}[b]{0.3\myFigWidth}
    \centering 
\resizebox {0.8\myFigSize} {!} {
\begin{tikzpicture}[
  vertex/.style={circle, draw, minimum size=8mm},
  tree/.style={isosceles triangle, fill=blue!20, draw, minimum height=1.5cm, minimum width=1cm, anchor=east},
  edge/.style={draw}
]

\node[vertex] (v) at (0,0) {$v$};
\node[vertex] (u) at (1.5,0) {$u$};
\node[vertex] (u0) at (3,0) {$u_0$};

\node[] (udots) at (1.5,-1.5) {$\cdots$};
\node[vertex, left=0.1cm of udots] (u1) {$u_1$};
\node[vertex, right=0.1cm of udots] (ul) {$u_\ell$};

\node[] (vdots) at (-1.5,0) {$\vdots$};
\node[tree, above=0.1cm of vdots] (v1) {$T_{v_1}$};
\node[, left=0.15cm of v1.apex] (oo_v1) {$v_1$};
\node[tree, below=0.1cm of vdots] (vk) {$T_{v_k}$};
\node[, left=0.15cm of vk.apex] (oo_vk) {$v_k$};

\draw[edge] (v) -- (u);
\draw[edge] (u) -- (u0);
\draw[edge] (u) -- (u1);
\draw[edge] (u) -- (ul);
\draw[edge] (v) -- (v1.apex);
\draw[edge] (v) -- (vk.apex);

\end{tikzpicture}
}
\caption{Case 3}
\label{fig:case_3}
     \end{subfigure}
        \caption{The unrooted tree $T$ before modification}
\end{figure}
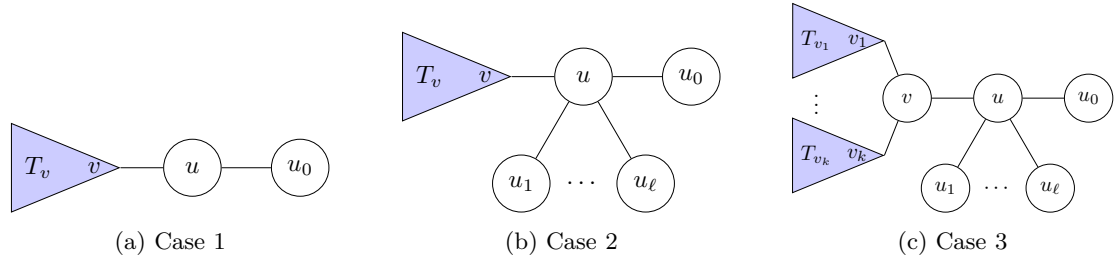

\textbf{Case 1: $\ell=0$ (see Figure~\ref{fig:case_1}).}
As $T$ is not the line graph and as $T_v$ has size $n-2$, by induction hypothesis, 
we either have $H_{T_v}(\alpha) > H_{L_{n-2}}(\alpha)$,	
or $T_v$ is a line graph rooted at a non-extremal vertex (and thus $C_{T_v}(\alpha) < C_{L_{n-2}}(\alpha)$).
If $C_{T_v}(\alpha) < C_{L_{n-2}}(\alpha)$, using \eqref{eq_H_T_rec} around edge $(v,u)$, 
as $\alpha (1+ \alpha) < 0$ (remind that $\alpha\in(-1,0)$), we get:
\begin{align*}
H_T(\alpha) & = H_{T_v}(\alpha) + H_{L_{2}}(\alpha) + 2 \alpha C_{T_v}(\alpha) (1 +\alpha) \\
	& >  H_{L_{n-2}}(\alpha) + H_{L_{2}}(\alpha) + 2 \alpha C_{L_{n-2}}(\alpha) (1 +\alpha) \\
	& = H_{L_n}(\alpha).
\end{align*}
If $C_{T_v}(\alpha) \geq C_{L_{n-2}}(\alpha)$ (thus we have $H_{T_v}(\alpha) > H_{L_{n-2}}(\alpha)$), 
using \eqref{eq_H_T_rec} around edge $(u,u_0)$, 
as $ C_{L_{n-1}}(\alpha) = 1+ \alpha C_{L_{n-2}}(\alpha)$, we get:
\begin{align*}
H_T(\alpha) & = H_{T\setminus\{u_0\}}(\alpha) + H_{L_{1}}(\alpha) + 2 \alpha (1 +\alpha C_{T_v}(\alpha)) \\
	& >  H_{L_{n-1}}(\alpha) + H_{L_{1}}(\alpha) + 2 \alpha (1+ \alpha C_{L_{n-2}}(\alpha)) \\
	& = H_{L_n}(\alpha).
\end{align*}
Thus, for all values of $C_{T_v}(\alpha)$, we get that $H_T(\alpha) > H_{L_n}(\alpha)$.

\textbf{Case 2: $\ell \geq 1$ and $C_{T_v}(\alpha) > 0$.}
Denote by $T_u$ the subtree of $T$ composed of the $\ell+2$ vertices $u,u_0,\cdots, u_\ell$ (see Figure~\ref{fig:case_2}),
and consider the tree $T'$ obtained from $T$ by replacing $T_u$ by a copy of $L_{2+\ell}$.
By induction hypothesis, we know that $H_{T_u}(\alpha) \geq H_{L_{2+\ell}}(\alpha)$.
Using \eqref{eq_H_T_rec} around edge $(v,u)$, 
as $1 +(\ell +1)\alpha < C_{L_{2+\ell}}(\alpha)$,
we get:
\begin{align*}
H_T(\alpha) 
& = H_{T_v}(\alpha) + H_{T_u}(\alpha) + 2 \alpha C_{T_v}(\alpha) (1 +(\ell +1)\alpha) \\
& >  H_{T_v}(\alpha) + H_{L_{2+\ell}}(\alpha) + 2 \alpha C_{T_v}(\alpha) C_{L_{2+\ell}}(\alpha) \\
& = H_{T'}(\alpha) ,
\end{align*}
and thus the tree $T$ does not minimize the Hosoya-Wiener polynomial.

\textbf{Case 3: $\ell \geq 1$ and $C_{T_v}(\alpha) \leq 0$.}
As $C_{T_v}(\alpha) \neq 1$, we know that $v$ is not a leaf.
Denote by $v_1, \cdots, v_k$ the neighbors of $v$ other than $u$,
and for each $i\in \llbracket 1, k\rrbracket$, denote by $T_{v_i}$ the subtree of $T_v$ rooted at $v_i$ (see Figure~\ref{fig:case_3}).
As $C_{T_v}(\alpha) = 1 + \alpha \sum_{i=1}^k C_{T_{v_i}}(\alpha) \leq 0$, 
we have that $\sum_{i=1}^k C_{T_{v_i}}(\alpha) \geq -1/\alpha > 0$,
and thus there exists $i\in \llbracket 1, k\rrbracket$ such that $C_{T_{v_i}}(\alpha) > 0$.
Without loss of generality, we assume for simplicity that $i=1$.

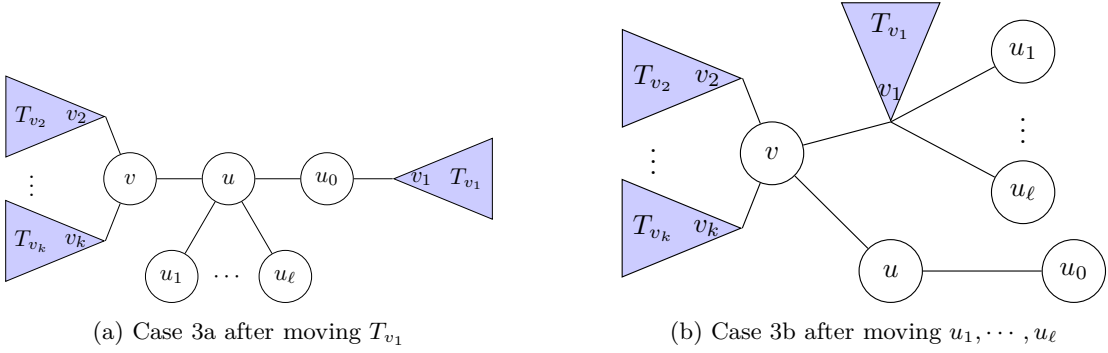
\begin{figure}[t]
     \centering
     \begin{subfigure}[b]{0.45\myFigWidth}
         \centering
\resizebox {1\myFigSize} {!} {
\begin{tikzpicture}[
  vertex/.style={circle, draw, minimum size=8mm},
  tree/.style={isosceles triangle, fill=blue!20, draw, minimum height=1.5cm, minimum width=1cm, anchor=east},
  treeright/.style={isosceles triangle, fill=blue!20, shape border rotate=180, draw, minimum height=1.5cm, minimum width=1cm, anchor=west},
  edge/.style={draw}
]

\node[vertex] (v) at (0,0) {$v$};
\node[vertex] (u) at (1.5,0) {$u$};
\node[vertex] (u0) at (3,0) {$u_0$};

\node[] (udots) at (1.5,-1.5) {$\cdots$};
\node[vertex, left=0.1cm of udots] (u1) {$u_1$};
\node[vertex, right=0.1cm of udots] (ul) {$u_\ell$};

\node[] (vdots) at (-1.5,0) {$\vdots$};
\node[tree, above=0.1cm of vdots] (v2) {$T_{v_2}$};
\node[tree, below=0.1cm of vdots] (vk) {$T_{v_k}$};

\node[treeright] (v1) at (4,0) {$T_{v_1}$};

\node[, right=0.15cm of v1.apex] (oo_v1) {$v_1$};
\node[, left=0.15cm of v2.apex] (oo_v2) {$v_2$};
\node[, left=0.15cm of vk.apex] (oo_vk) {$v_k$};

\draw[edge] (v) -- (u);
\draw[edge] (u) -- (u0);
\draw[edge] (u) -- (u1);
\draw[edge] (u) -- (ul);
\draw[edge] (v) -- (v2.apex);
\draw[edge] (v) -- (vk.apex);
\draw[edge] (u0) -- (v1.apex);

\end{tikzpicture}
}
\caption{Case 3a after moving $T_{v_1}$}
\label{fig:case_3a}
     \end{subfigure}
     \hfill
     \begin{subfigure}[b]{0.45\myFigWidth}
    \centering 
\resizebox {1\myFigSize} {!} {
\begin{tikzpicture}[
  vertex/.style={circle, draw, minimum size=8mm},
  tree/.style={isosceles triangle, fill=blue!20, draw, minimum height=1.5cm, minimum width=1cm, anchor=east},
  treesouth/.style={isosceles triangle, fill=blue!20, shape border rotate=-90, draw, minimum height=1.5cm, minimum width=1cm, anchor=south},
  subtree/.style={trapezium, draw, fill=blue!20, shape border rotate=180, trapezium angle=70, minimum height=0.8cm, minimum width=0.4cm, anchor=south},
  edge/.style={draw}
]

\node[vertex] (v) at (0,0) {$v$};
\node[treesouth] (v1) at (1.5,0.4) {$T_{v_1}$};
\node[vertex] (u) at (1.5,-1.5) {$u$};
\node[vertex, right=1.5 of u] (u0) {$u_0$};

\node[, right=1.5 of v1.apex] (udots) {$\vdots$};
\node[vertex, above=0.1cm of udots] (u1) {$u_1$};
\node[vertex, below=0.1cm of udots] (ul) {$u_\ell$};

\node[] (vdots) at (-1.5,0) {$\vdots$};
\node[tree, above=0.1cm of vdots] (v2) {$T_{v_2}$};
\node[tree, below=0.1cm of vdots] (vk) {$T_{v_k}$};

\node[, above=0.15cm of v1.apex] (oo_v1) {$v_1$};
\node[, left=0.15cm of v2.apex] (oo_v2) {$v_2$};
\node[, left=0.15cm of vk.apex] (oo_vk) {$v_k$};

\draw[edge] (v) -- (v1.apex);
\draw[edge] (u) -- (v);
\draw[edge] (u) -- (u0);
\draw[edge] (v1.apex) -- (u1);
\draw[edge] (v1.apex) -- (ul);
\draw[edge] (v) -- (v2.apex);
\draw[edge] (v) -- (vk.apex);

\end{tikzpicture}
}
\caption{Case 3b after moving $u_1,\cdots, u_\ell$}
\label{fig:case_3b}
     \end{subfigure}
        \caption{The unrooted tree $T'$ in Case 3 after modification}
\end{figure}

\textbf{Case 3a: $\ell \geq 1$ and $C_{T_v}(\alpha) \leq 0$ and $\sum_{i=2}^k C_{T_{v_i}}(\alpha) >0$.}
Consider the tree $T'$ obtained from $T$ by replacing edge $(v_1,v)$ by $(v_1,u_0)$,
that is grafting $T_{v_1}$ on $u_0$ (see Figure~\ref{fig:case_3a}).
Remark that going from $T$ to $T'$, the subtrees $T_{v_1}$ and $T\setminus T_{v_1}$ do not change, 
the distance between $v_1$ and $v_i$ for $i\in \llbracket 2, k\rrbracket$ goes from $2$ to $4$,
and that $v_1$ is still at distance $1$ from a leaf of the star graph formed by vertices $\{ v,u, u_0, \dots, u_\ell\}$.
Thus, as $\alpha\in(-1,0)$, we have:
\begin{equation*}
H_T(\alpha) - H_{T'}(\alpha)
= 2 C_{T_{v_1}}(\alpha) \left(  \sum_{i=2}^k C_{T_{v_i}}(\alpha) \right) (\alpha^2 - \alpha^4) > 0 ,
\end{equation*}
and thus the tree $T$ does not minimize the Hosoya-Wiener polynomial.

\textbf{Case 3b: $\ell \geq 1$ and $C_{T_v}(\alpha) \leq 0$ and $\sum_{i=2}^k C_{T_{v_i}}(\alpha) \leq 0$.}
In particular, we get that $C_{T_{v_1}}(\alpha) \geq -1/\alpha \geq 1$.
Consider the tree $T'$ obtained from $T$ by replacing edge $(u_i,u)$ by $(u_i,v_1)$ for all $i\in \llbracket 1,\ell \rrbracket$,
that is grafting leaves $u_1,\dots, u_\ell$ to $v_1$ (see Figure~\ref{fig:case_3b}).
Remark that going from $T$ to $T'$, the subtree $T\setminus \{ u_1, \cdots, u_\ell \}$ does not change, 
and for for $i,j \in \llbracket 1, \ell \rrbracket$, the distance between $u_i$ and $u_j$ (resp. $v$) does not change,
and the distance between $u_i$ and $v_1$ (resp. $u$) goes from $3$ to $1$ (resp. from $1$ to $3$).
Thus, as $\alpha\in(-1,0)$, we have:
\begin{align*}
H_T(\alpha) - H_{T'}(\alpha)
& = 2 \ell C_{T_{v_1}}(\alpha) (\alpha^3 - \alpha) + 2 \ell (1+\alpha) (\alpha - \alpha^3) \\
& \geq 2 \ell (\alpha^3 -\alpha) + 2 \ell (1+\alpha) (\alpha - \alpha^3) \\
& \geq 2 \ell (\alpha^2 -\alpha^4) > 0, 
\end{align*}
and thus the tree $T$ does not minimize the Hosoya-Wiener polynomial.
\end{proof}

\providecommand{\MRhref}[2]{%
  \href{http://www.ams.org/mathscinet-getitem?mr=#1}{#2}
}
\providecommand{\href}[2]{#2}
\providecommand{\MR}[1]{\href{http://www.ams.org/mathscinet-getitem?mr=#1}{MR #1}}
\providecommand{\ARXIV}[1]{\href{https://arxiv.org/abs/#1}{arXiv:#1}}


\begin{thebibliography}{99}
\bibliographystyle{APT}
\footnotesize

\bibitem{IntroGWTrees}
{\sc Abraham, R. and Delmas, J-F.} (2020). 
An introduction to {{Galton-Watson}} trees and their local limits.
  \ARXIV{1506.05571}

\bibitem{athreyaCoalescenceCriticalSubcritical2012}
\sc{Athreya, K.~B.} (2012). 
Coalescence in {{Critical}} and {{Subcritical Galton-Watson Branching}}.
	{\em Journal of Applied Probability}
  \textbf{49}, no.~3, 627--638.
\MR{3012088}

\bibitem{athreyaCoalescenceRecentRapidly2012}
\sc{Athreya, K.~B.} (2012).  
Coalescence in the recent past in rapidly growing populations.
  {\em Stochastic Processes and their Applications} 
  \textbf{122}, no.~11,  3757--3766.
\MR{2965924}

\bibitem{athreyaLimitTheoremsPositive1998}
\sc{Athreya, K.~B. and Kang, H.-J.} (1998).
{Some limit theorems for positive recurrent branching {{Markov}} chains: {{I}}}.
	{\em Advances in Applied Probability}
	\textbf{30}, no.~3, 693--710.
\MR{1663545}

\bibitem{athreyaLimitTheoremsPositive1998a}
\sc{Athreya, K.~B. and Kang, H.-J.} (1998). 
{Some limit theorems for positive recurrent branching {{Markov}} chains: {{II}}}, 
  {\em Advances in Applied Probability} 
  \textbf{30}, no.~3, 711--722.
  \MR{1663545}

\bibitem{athreyaBranchingProcesses1972}
\sc{Athreya, K.~B. and Ney, P.~E.} (1972). 
{\em Branching {{Processes}}}, {Springer
  Berlin Heidelberg}, {Berlin, Heidelberg}.
\MR{0373040}

\bibitem{bansayeAncestralLineagesLimit2019}
\sc{Bansaye, V.} (2019). 
{Ancestral {{Lineages}} and {{Limit Theorems}} for
  {{Branching Markov Chains}} in {{Varying Environment}}}, 
  {\em Journal of Theoretical Probability} 
  \textbf{32}, no.~1, 249--281.
\MR{3908914}

\bibitem{casablancaDistanceEccentricSequences2019}
\sc{Casablanca, R. M. and Dankelmann, P.} (2019).
{Distance and {{Eccentric}} sequences to bound the {{Wiener}} index, {{Hosoya}} polynomial and the average eccentricity in the strong products of graphs}, 
  {\em Discrete Applied Mathematics}
  \textbf{263}, 105--117.
\MR{3956040}

\bibitem{delmasDetectionCellularAging2010}
\sc{Delmas, J.-F. and Marsalle, L.} (2010). 
{Detection of cellular aging in a {{Galton}}{\textendash}{{Watson}} process}, 
	{\em Stochastic Processes and their Applications}
	\textbf{120}, no.~12, 2495--2519.
\MR{2728175}

\bibitem{GuyonLimitTheorem}
\sc{Guyon, J.} (2007).
{Limit theorems for bifurcating {{Markov}} chains.
  {{Application}} to the detection of cellular aging}, 
  {\em The Annals of Applied Probability}
  \textbf{17}, no.~5-6.
\MR{2358633}

\bibitem{rudinFunctionalAnalysis1996}
\sc{Rudin, W.} (1996). 
{\em Functional analysis}, 2nd~edn.
International Series in Pure and Applied Mathematics, {McGraw-Hill}, {Boston, Mass.}.
\MR{0365062}

\bibitem{tianSharpBoundsNormalization2013}
\sc{Tian, D. and Choi, K. P.} (2013).
{Sharp Bounds and Normalization of Wiener-Type Indices}, 
  {\em PLoS ONE}
  \textbf{8}, no.~11, e78448.

\bibitem{wagnerDistancebasedGraphInvariants2013}
\sc{Wagner, S., Wang, H. and Zhang, X.-D.} (2013).
{Distance-based graph invariants of trees and the {{Harary}} index}, 
  {\em Filomat}
  \textbf{1}, 41--50.
\MR{3243897}




\end{thebibliography}
\end{document}